\documentclass{amsart}
\usepackage{amssymb}
\usepackage{graphicx}
\usepackage{epstopdf}
\usepackage{amscd}

\newtheorem{thm}{Theorem}[section]
\newtheorem{cor}[thm]{Corollary}
\newtheorem{lem}[thm]{Lemma}
\newtheorem{prop}[thm]{Proposition}

\theoremstyle{definition}
\newtheorem{defn}[thm]{Definition}
\theoremstyle{remark}
\newtheorem{rem}[thm]{Remark}

\numberwithin{equation}{section}
\newcommand{\To}{\longrightarrow}

\newcommand{\inv}{^{-1}}
\newcommand{\C}{\mathbb C}
\newcommand{\Z}{\mathbb Z}
\newcommand{\R}{\mathbb R}

\newcommand{\x}{\times}

\newcommand{\pr}{\operatorname{pr}}

 \makeatletter 
 \def \underbracket { %
 \@ifnextchar[{\@underbracket}{\@underbracket[\@bracketheight]}%
 } 
 \def\@underbracket[#1]{ %
 \@ifnextchar[{\@under@bracket[#1]}{\@under@bracket[#1][0.4em]} %
 } 
\def\@under@bracket[#1][#2]#3{
\mathop{\vtop{\m@th\ialign{##\crcr$\hfil\displaystyle{#3}\hfil$ %
\crcr\noalign{\kern3\p@\nointerlineskip}\upbracketfill
{#1}{#2} 
\crcr\noalign{\kern3\p@}}}}\limits} 
\def\upbracketfill#1#2{$\m@th\setbox\z@\hbox{$\braceld$} 
\edef\@bracketheight{\the\ht\z@}\bracketend{#1}{#2} 
\leaders\vrule\@height#1\@depth\z@\hfill
\leaders\vrule\@height#1\@depth\z@\hfill 
\bracketend{#1}{#2}$} 
\def\bracketend#1#2{\vrule height#2 width#1\relax} 
\makeatother

\begin{document}

\title
{A mirror symmetric solution to the quantum Toda lattice}
\author{Konstanze Rietsch}%
\address{King's College London, UK}%
\email{konstanze.rietsch@kcl.ac.uk}%


\thanks{
The author is supported by EPSRC advanced fellowship
EP/S071395/1.}


\date{March 22, 2011}
\begin{abstract}
We give a representation-theoretic proof of
a conjecture from \cite{Rie:MSgen} providing
integral formulas for solutions to the quantum Toda lattice 
in general type. 
This result generalizes work of Givental for $SL_n/B$
in a uniform way to arbitrary type,
and can be interpreted as a kind of mirror theorem for the full flag 
variety $G/B$. We also 
prove the existence of a totally positive and totally 
negative critical point of the
`superpotential' in every mirror fiber.  
\end{abstract}
\maketitle
\section{Introduction}
\subsection{}
In \cite{Rie:MSgen} we introduced a conjectural
`mirror datum' $(Z_P, \omega,\mathcal F_P)$
associated to a general flag variety $G/P$.
The main goal of this paper is to show, in the full 
flag variety case, that associated integrals
\begin{equation}\label{e:S}
S_\Gamma(h)=\int_{[\Gamma_h]}  
e^{\mathcal F_B}\ \omega_h , 
\end{equation}
defined in terms of the mirror datum 
$(Z=Z_B, \omega,\mathcal F_B)$ of $G/B$,
are annihilated by the quantum Toda 
Hamiltonian associated to $G^\vee$~\cite{Kos:qToda}. 
Here $h\in\operatorname{ Lie}(T^\vee)$ for a maximal torus $T^\vee$ of $G^\vee$.

To give an idea of the definitions,
after trivializing the mirror family $Z=(Z_h)_h$
the form $\omega_h$ in the formula above
can be understood as a particular torus-invariant 
rational section $\omega$ of the canonical bundle of the Langlands dual flag variety 
$G^\vee/B^\vee$. The `mirror fiber' $Z_h$,
inside of which $[\Gamma_h]$ is a middle dimensional cycle,
may be identified with the affine subvariety of $G^\vee/B^\vee$ 
where $\omega$ is regular. (Explicitly, it is an intersection 
of opposite big cells, often denoted $\mathcal R^\vee_{1,w_0}$). The definition of 
the superpotential $\mathcal F_B$ comes down to writing an element of 
$\mathcal R^\vee_{1,w_0}$ from two points of view, 
with an additional twist by $e^h$, and adding
up the `coordinates'. For example for $G=SL_2$ 
we have 
$\mathcal R^\vee_{1,w_0}=\C P^1\setminus \{0,\infty\}$, 
and the point with
coordinate $u$ when viewed from the direction of $0$, say, has 
coordinate  $e^h/u$ (after twisting)
when viewed from $\infty$. This 
recovers the familiar formula  $u+\frac{e^h}u$
for the superpotential of $\C P^1$ (eg. \cite{
Giv:ICM, EY:94}).

For more detailed definitions of what is involved in 
defining the integrals \eqref{e:S} in general see Section~6 (for the mirror family 
$Z_B$ and superpotential $\mathcal F_B$),
Section~5 (for the holomorphic $n$-form $\omega_h$),
and Section~7 for a discussion of the integration 
cycles $[\Gamma_h]$.

\subsection{}
We recall that by work of B.~Kim \cite{Kim:qCohGB}, 
the quantum Toda equations of $G^\vee$ are  
the `quantum differential equations
 of $G/B$' in the sense of Givental, whose symbols 
recover the relations in the quantum cohomology 
ring. See also \cite{Braverman:InstantonCounting}.
In  \cite{Rie:MSgen} we showed how the mirror datum 
$(Z_P, \omega,\mathcal F_P)$
can be used to recover the 
quantum cohomology ring $qH^*(G/P)_{(q)}$ with 
quantum parameters inverted, in its presentation
due to Dale Peterson.  
Therefore the main result in the present paper can be 
seen as a kind of `quantization' of a result from \cite{Rie:MSgen}. 
Moreover, our main result here is a special case
of Conjecture~8.2 in \cite{Rie:MSgen}.

\subsection{}
In type $A$, integral solutions of the form \eqref{e:S} to the quantum 
Toda lattice were obtained by Givental
\cite{Giv:MSFlag} using ingenious and very explicit coordinates, in 
what he called a `mirror theorem' for $SL_n/B$.
The general mirror family introduced in \cite{Rie:MSgen} 
was very much inspired by Givental's construction, but
is in contrast given in Lie theoretic terms and without preferred choice 
of coordinates.

Givental's construction of solutions to the 
quantum Toda lattice was  also
recently revisited by 
Gerasimov, Kharchev, Lebedev and Oblezin 
\cite{GKLO:GaussGivental} (GKLO),
who gave a new proof of Givental's type $A$  
result using Kostant's Whittaker model. 
Their proof has some  features in common with 
our construction \cite{Rie:MSgen} but still relied 
in an essential way on the use of 
Givental's special coordinates.  This work is
partly inspired by ideas from \cite{GKLO:GaussGivental}.

Integral solutions to the quantum Toda
lattice different from the ones in the present
paper and much background information
can be found in the nice 
book chapter of Semenov-Tian-Shansky
\cite{STS:bookchapter}.
 Comparison with another work of 
Gerasimov, Lebedev and Oblezin 
\cite{GLO:NewIntRep} is given in 
Section~\ref{s:GLO}.

\subsection{}
The larger part of this paper is concerned with the
proof the main theorem, Theorem~\ref{t:main}, 
which says that the integrals \eqref{e:S} with 
our Lie theoretic mirror datum provide solutions
to the quantum Toda lattice in general type. 
Kostant's classical result about the construction 
of solutions to the quantum Toda lattice using 
Whittaker modules is central to this work and
is recalled in Section~\ref{s:qToda}. Namely, a 
key step in our proof is the construction of two 
particular dually paired Whittaker modules 
with zero infinitesimal character, defined 
on a space of functions on the mirror, and
the computation of their Whittaker 
vectors, see Sections~9 and 10.

The comparison of the superpotential $\mathcal F_B$
in our integral formula with the original one introduced by 
Givental in type A is explained in detail in the
precursor to this paper \cite[Section~9]{Rie:MSgen}, where the
Lie theoretic superpotential we use here was 
originally introduced.
It is simply the restriction to a particular coordinate 
patch isomorphic to $(\C^*)^N$ (followed by passing
to logarithmic coordinates).

In Section~\ref{s:contours} we discuss how 
we want to choose the $[\Gamma_h]$ in \eqref{e:S}.
We also describe explicitly two distinguished such choices 
of families of integration cycles, for which the 
corresponding integrals can be made canonical
in a sense explained in Section~\ref{s:mainthm}. 
The one family of cycles 
can be thought of as running through the
`totally negative parts' of the mirror fibers
and is made up of non-compact cycles.
The other  family consists of cycles 
that are compact.  

Additionally, in Section 11, we study the
critical points of $\mathcal F_B$, and prove the
existence of a totally positive and a totally
negative critical point (at any fixed positive
value of the quantum parameters). This 
partially generalizes an earlier result in type $A$, see 
\cite{Rie:TotPosGBCKS}.

\vskip .2cm

\subsection{}\label{s:GLO}
After submission of this paper I found out that a
further paper of Gerasimov, Lebedev and Oblezin 
on solutions to the quantum Toda lattice for Lie groups
of classical types covers similar ground. Their work 
\cite{GLO:NewIntRep} was done at a similar time 
and appeared as preprint while the preprint version 
\cite{Rie:QToda} was just being prepared.

While we consider the zero infinitesimal character case, 
the paper
\cite{GLO:NewIntRep} 
lets the infinitesimal character vary to obtain eigenfunctions, 
not just zeros, of the classical type quantum Toda lattice. 
I believe it is straightforward to check that their 
formula \cite[Proposition 2.1]{GLO:NewIntRep} for these 
eigenfunctions agrees with one conjectured earlier in 
\cite[Conjecture~8.2]{Rie:MSgen}, 
and therefore their results overlap with results 
in the present paper. The overlap being the classical type case 
formula for zeros of the quantum Toda lattice.

Despite this overlap in content our present paper and
\cite{GLO:NewIntRep} are 
quite dissimilar papers in some ways. For example I 
consider a different setting for the integration cycles, 
as well as making sure everything works in general type 
and is expressed in the (to me) most natural way in 
terms of Lie theory. I also prove existence of totally
positive and totally negative critical points of the
superpotential. On the other hand Gerasimov, Lebedev
and Oblezin
give the more general eigenfunctions in their setting, 
and go on to choose clever coordinates and express 
their integrands explicitly in all classical types. 
This gives a very nice additional result, which also
generalizes Givental's work more directly.

While reading the paper \cite{GLO:NewIntRep} I
noticed that the classical type assumption in 
\cite[Proposition 2.1]{GLO:NewIntRep} 
is used in a part of the proof dealing 
with holomorphic volume 
forms \cite[Lemma~3.1]{GLO:NewIntRep}, 
and possibly nowhere else. 
However the relevant lemma on holomorphic
volume forms was already 
carefully stated and proved in general 
type in \cite[Section~7]{Rie:MSgen}. 
Possibly, therefore, combining the work in \cite{GLO:NewIntRep}
with this earlier result, 
one may be able to infer our conjecture 
 for eigenfunctions of the quantum 
differential equations \cite[Conjecture~8.2]{Rie:MSgen}
also for the exceptional groups.
  
\subsection{} From now on we swap the roles of 
 $G$ and its Langlands dual 
 compared to \cite{Rie:MSgen}, as the 
 `$A$-model' will not enter into the picture 
 much anymore. So the goal is to construct solutions 
 to the quantum Toda lattice for $G$.

\vskip .5cm

\noindent{\bf Acknowledgments.} 
I thank Ian Grojnowski for his hospitality
in 2006/07 when I was his visitor at the 
University of Cambridge and preparing this work. 
I am grateful to Dale Peterson for his inspirational
lectures \cite{Pet:QCoh}. I also thank Tom Coates for 
showing me the paper of F.~Pham. 
Finally, I thank one 
anonymous referee for asking me more
questions about the integration cycles, and another 
for prodding me to look at the related work of 
Gerasimov, Lebedev and Oblezin.  
\section{Notation and Preliminaries}

We refer to \cite{Springer:book,Knapp:LieGroupBook} for
background on algebraic groups and basic representation theory.
Let $G$ be a simple, simply connected algebraic group over $\C$
of rank $n$ with split real form.  
We fix opposite Borel subgroups $B=B_-$ and $B_+$ with unipotent radicals 
$U_-$ and $U_+$, respectively. Assume that $B_+$ and $B_-$ are also 
defined over $\R$ and the maximal torus $T=B_+\cap B_-$ is split. Let
$W=N_G(T)/T$ denote the Weyl group.

Let $\mathfrak g$ be the Lie algebra of $G$. We denote by  
$\mathfrak b_-,\mathfrak b_+,\mathfrak u_-,\mathfrak u_+, 
\mathfrak h$ the Lie algebras of $B_-,B_+,U_-,U_+$ and $T$, respectively.
Let $\mathfrak h_\R$ be the real form of $\mathfrak h$ and $\mathfrak g_\R$
the real form of $\mathfrak g$.

We set $I=\{1,\dotsc,n \}$ and choose  $\{\alpha_i \ | \ i\in I\}$ to be the set of 
simple roots associated to the positive 
Borel $B_+$. We may view the $\alpha_i$ as elements of $\mathfrak h^*$
or as characters on $T$. The coroot associated to $\alpha_i$ is  
denoted $\alpha_i^\vee$ and gives a 
1-parameter subgroup of $T$ denoted $s\mapsto s^{\alpha_i^\vee}$. 

Corresponding to the positive and negative simple roots we have 
Chevalley generators $e_i, f_i\in \mathfrak g$
and one parameter subgroups, 
 \begin{align*}
x_i(t):=\exp(t e_i), &\qquad y_i(t):=\exp(t f_i), \qquad {t\in \C},
 \end{align*}
in $U_+$ and $U_-$, respectively. Let 
\begin{equation}\label{e:Weyl}
\dot s_i=x_i(-1)y_i(1)x_i(-1).
\end{equation}
This element represents a simple reflection $s_i$ in the Weyl group $W$.

For $w\in W$, a representative
$\dot w\in G$ is defined by $\dot w=\dot s_{i_1}\dot
s_{i_2}\cdots \dot s_{i_m}$, where $ s_{i_1} s_{i_2}\cdots
s_{i_m}$ is a (any) reduced expression for $w$. The length $m$ of
a reduced expression for $w$ is denoted by $\ell(w)$.

Let $\left<\ ,\ \right >$ be the $W$-invariant inner product 
on $\mathfrak h^*$ such that $\left < \alpha,\alpha\right>=2$ 
for any long root $\alpha$. We also denote the corresponding 
inner product on $\mathfrak h$ in the same way, by $\left<\ ,\ \right>$. 

Let $\mathcal U(\mathfrak g)$ be the universal enveloping algebra of $\mathfrak g$, 
and $\mathcal Z(\mathfrak g)$ its center, and let 
$$
\gamma:\mathcal Z(\mathfrak g)\overset\sim\To \mathbb C[\mathfrak h^*]^{W},
$$ 
be the  Harish-Chandra homomorphism to the ring of $W$-invariant
regular functions on $\mathfrak h^*$.
Define $c_2\in\mathcal Z(\mathfrak g)$ by 
$\gamma(c_2)(a)=\left<a,a\right >$ for $a\in\mathfrak h^*$.

For any integral dominant weight $\lambda$ we have an irreducible 
representation $V(\lambda)$ of $G$. In each $V(\lambda)$ 
let us fix a highest weight vector $v^+_{\lambda} $.
Then for 
any $v\in V(\lambda)$ and extremal weight vector $\dot w\cdot v^+_{\lambda}$ 
we have the coefficient $\left<v,\dot w\cdot v^+_{\lambda}\right>\in \C$ defined by
\begin{equation*}
v=\left<v,\dot w\cdot v^+_{\lambda}\right>\dot w\cdot  v^+_{\lambda}+
 \text{other weight space summands.}
\end{equation*}
We define $v^-_{\lambda}:=\dot w_0\cdot v^+_{\lambda}$.
The most important
choices for $\lambda$ are the fundamental weights $\omega_i$, where $i\in I$, 
and $\rho:=\sum_{i\in I}\omega_i$.

Let $\mathcal B$ be the set of Borel subgroups in $G$ with the conjugation action,
$$
g\cdot B:=g B g\inv \qquad g\in G,\ B\in \mathcal B.
$$
Then we may identify $\mathcal B$ with the flag variety $G/B_-$ in the usual
way, by sending $gB_- $ to $ g\cdot B_-$.
Consider the Zariski open subset of $G/B_-$ given by the intersection of two 
opposed big Bruhat cells
\begin{equation*}
\mathcal R_{1,w_0}:=(B_+ B_-
\cap B_-\dot w_0 B_-)/B_-  ,
\end{equation*}
which can also be written more symmetrically as
$$
\mathcal R_{1,w_0}=B_+\cdot B_-\cap B_-\cdot B_+.
$$
We note that $\mathcal R_{1,w_0}$ is the complement of 
an anti-canonical divisor (namely, explicitly, the divisor defined
by the inverse to the section $\omega$ of
the {\it canonical} bundle given
in Proposition~\ref{p:omega}).

\section{The Toda lattice}\label{s:Toda}

The classical Toda lattice is an integrable system with Hamiltonian
$$
H(x_i, p_i)=\frac 1 2\sum p_i^2 +\sum e^{x_i-x_{i+1}},
$$
which was solved by Moser \cite{Moser:Toda} in the 1970's.
This integrable system was generalized to arbitrary $G$ by Kostant \cite{Kos:Toda},
using phase space $T^*(T_{\operatorname{ad}})\hat=T_{\operatorname{ad}}\x\mathfrak h^*$, the cotangent bundle 
to the adjoint torus $T_{\operatorname{ad}}$, and Hamiltonian
$$
H(t,h^*)=\frac 1 2\left<h^*,h^*\right> + \sum_i\alpha_i(t),
$$
where the simple root $\alpha_i$ is understood as a character
 on $T_{\operatorname{ad}}$. 
Kostant then solved
this system using a carefully chosen  embedding of the phase space into the
dual space $\mathfrak g^*$ of $\mathfrak g$, given by
\begin{equation}\label{e:phasespace}
(t,h^*)\mapsto F + h^* +\sum_{i\in I} {\alpha_i}(t) {f_i}^*. 
\end{equation}
Here ${e_i}^*,{f_i}^*\in \mathfrak g^*$ are defined to take value $1$ on 
$e_i$ and $f_i$, respectively, and vanish on all other 
weight spaces of $\mathfrak g$. Also $F=\sum e_i^*$. Note that
$\mathfrak h^*, \mathfrak b_-^*$ are subspaces of 
$\mathfrak g^*$. 
The image of the phase space is in fact the translate by $F$
of a coadjoint orbit in $\mathfrak b_-^*$. Moreover, the 
Toda Hamiltonian now appears naturally as the restriction 
of (essentially) the Killing form, and 
a full set of Poisson-commuting constants of motion comes 
from restricting the remaining generators of $\C[\mathfrak g^*]^G$.

\subsection{The (Givental-)Kim  presentation of $qH^*(G^\vee/B^\vee)$}
In \cite{Kim:qCohGB}  Kim described the relations of the 
small quantum cohomology
ring of $G^\vee/B^\vee$ in terms of constants of motion of the Toda lattice
associated to $G$, generalizing his joint work with Givental \cite{GiKi:FlTod}
in type~A.  Formally, this goes as follows.
Let
$\mathcal A^*$ denote the image of the embedding~\eqref{e:phasespace},  
$$
\mathcal  A^*= F+\mathfrak h^*+\sum(\C\setminus\{ 0\}) \ {f_i}^*,
$$
and let $\phi_1,\dotsc, \phi_n$ be a  set of homogeneous generators for 
$\C[\mathfrak h^*]^W$, following Chevalley.
Consider the `diagonalization' map 
$$
\Sigma:\mathcal A^* \to \mathfrak h^*/W.
$$
That is, $\Sigma$ corresponds to 
$$
\C[\mathfrak h^*]^W \overset\sim\to\C[\mathfrak g^*]^{G}\hookrightarrow \C[\mathfrak g^*]\to \C[\mathcal A^*],
$$
where the first map is given by Chevalley's restriction theorem, and the 
third map comes from the inclusion $\mathcal A^*\hookrightarrow\mathfrak g^*$.
Then for the ring $qH^*(G^\vee/B^\vee)[q_1\inv,\dotsc, q_{n}\inv]$ 
with quantum parameters inverted, Kim's presentation takes the form 
of an isomorphism 
$$
qH^*(G^\vee/B^\vee,\C)[q_1\inv,\dotsc, q_n\inv]\overset\sim\to  \C\left [\mathcal A^*\underset{\mathfrak h^*/W}\x \{0\}\right ]=\C[\mathcal A^*]/(\Sigma_1,\dotsc, \Sigma_n),
$$
where the $\Sigma_i:=\Sigma^*(\phi_i)$ are precisely Kostant's constants of motion
for the Toda lattice. In this presentation the usual 
Chern class generators $x_i$ of the quantum cohomology ring appear via the isomorphism
$H^2(G^\vee/B^\vee)\cong (\mathfrak h^\vee)^*$ and 
the identification of $(\mathfrak h^\vee)^*$ with 
functions on the $\mathfrak h^*$-part of $\mathcal A^*$.  Explicitly, the
$x_i$ and $q_i$ are identified with the coordinates on 
the degenerate leaf in $\mathcal A^*$ via
$$
h^*=-\sum_{i\in I}(x_1+\dotsc+ x_i)\alpha_i,\qquad t=\prod_{i\in I}  (-q_i)^{\omega_i^\vee},
$$  
where $h^*$ and $t$ are as in \eqref{e:phasespace}. 

\subsection{Peterson's presentation of $qH^*(G^\vee/B^\vee)$}
The mirror symmetric approach to the quantum cohomology rings \cite{Rie:MSgen} is more
closely related to an alternative presentation of $qH^*(G^\vee/B^\vee)$ due to Dale 
Peterson \cite{Pet:QCoh}, see also \cite{Kos:QCoh,Kos:QVandermonde,
Rie:QCohPFl, Rie:JAMSerr,Rie:MSgen}. 
Namely consider the closed subvariety $Y$ in $G/B_-$ defined by
$$
Y=\left \{gB_-\ |\ (g\inv\cdot F)|_{[\mathfrak u_-,\mathfrak u_-]}=0\right \}, 
$$
which is called the Peterson variety. 
Then in Peterson's presentation the quantum cohomology ring appears
as ring of regular functions on an open stratum of $Y$. Explicitly, in the
case where we localize at the quantum parameters,
$$
qH^*(G^\vee/B^\vee,\C)[q_1\inv,\dotsc, q_{n}\inv]\cong\C\left[\ Y\underset{\ G/B_-}\x \mathcal R_{1,w_0}\right].
$$
This presentation relates to Kim's presentation by the map taking 
$uB_-\mapsto u\inv\cdot F$ for $u\in U_+\cap B_-\dot w_0 B_-$,
which defines a morphism
$$ 
Y\underset{\ G/B_-}\x \mathcal R_{1,w_0}\to  \mathcal A^*\underset{\mathfrak h^*/W}\x\{0\}.
$$
That this is an isomorphism follows from work of Kostant, see~\cite{Kos:PolReps}.

Peterson's presentation via his variety $Y$ has some surprising advantages over Kim's presentation. Namely, the Peterson variety also sees the quantum cohomology rings of partial flag 
varieties $G^\vee/P^\vee$, by intersection with associated smaller Bruhat cells, and it provides presentations over the integers in general type \cite{Pet:QCoh}.

\section{The Quantum Toda lattice}\label{s:qToda}
The quantum Toda Hamiltonian we are interested in is the differential operator on $\mathfrak h$ defined by
\begin{equation}\label{e:QTodaHamiltonian}
\mathcal H_G=\frac 1 2  \Delta\ -\frac 1{z^2}
\sum_{i\in I} e^{\alpha_i},
\end{equation}
where $\Delta$ is the Laplace operator associated to the $W$-invariant inner product 
$\left <\ ,\ \right>$ on $\mathfrak h$, and the 
 $\alpha_i$ are the simple roots inside $\mathfrak h^*$.
This can be understood as a quantization via the orbit method (for a subgroup of $B_-$),
which gives rise to Kostant's `Whittaker model' for the centre of the universal enveloping
algebra inside $\mathcal U(\mathfrak b_-)$. The quantum Toda lattice was
constructed, analysed and solved by Kostant in \cite{Kos:qToda, Kos:Whittaker}. 

As Kazhdan and Kostant observed \cite{Etingof:QWhittaker}, 
the operator $\mathcal H_G$ can be obtained
from the Laplace operator on $G$ itself by restricting to appropriate `Whittaker functions'. 
This approach is the analogue of the embedding of $T^*(T_{\operatorname {ad}})$ into a `tridiagonal' part 
of $\mathfrak g^*$ in order to realize the classical Toda Hamiltonian as coming from 
an invariant quadratic form.
The higher Casimirs generalizing the Laplace operator provide quantum integrals of motion, 
in analogy with the Poisson commuting constants of motion obtained in the classical case as generators of
the ring of invariants (the direct translation from quantum integrals to classical integrals of motion being given by the Harish-Chandra homomorphism). 

In Section~\ref{s:Whittaker} we will review and adapt to our situation this construction 
of the quantum Toda lattice, as
Whittaker functions will turn out to appear very naturally in the mirror model. 
The main goal of the present paper is to connect Kostant's representation theoretic approach 
to the quantum Toda lattice with the general type $B$-model of $G^\vee/B^\vee$, and use it to prove 
the mirror conjecture from \cite{Rie:MSgen} for full flag varieties. 

\subsection{Whittaker functions and the quantum Toda Hamiltonian}\label{s:Whittaker}
Let $\chi_+:\mathfrak u_+\to\C$ and $\chi_-:\mathfrak u_-\to \C$ be Lie algebra homomorphisms. 
Then $\chi_+$ and $\chi_-$ are determined by $\chi_+(e_i)=:\chi_+^{(i)}$, and $\chi_-(f_i)=:\chi_-^{(i)}$,
respectively. Namely,
\begin{eqnarray*}
\chi_+&=&\sum_{i\in I}\chi_+^{(i)}e_i^*,\\
\chi_-&=&\sum_{i\in I}\chi_-^{(i)}f_i^*.
\end{eqnarray*}
By abuse of notation we denote by $e_i^*$ also the map $U_+\to\C$
which takes $u$ to the coefficient of $e_i$ in the series expansion
$u=\exp(n)=1+\sum e_i^*(n)e_i + \dotsc $, and similarly for $f_i^*$.    
The holomorphic characters of $U_+$ and $U_-$, respectively, 
corresponding to $\chi_+$ and $\chi_-$ are denoted by
\begin{eqnarray*}
e^{\chi_+}:&U_+\to \C^*, \\
e^{\chi_-}:&U_-\to \C^*. 
\end{eqnarray*}
If for all $i$ the coefficient $\chi_+^{(i)}\ne 0$, then  $e^{\chi_+}$ and also $\chi_+$ are 
called non-degenerate. Analogously for $\chi_-$.

For our purposes the `Whittaker functions' will not be functions on $G$ but 
functions on a universal cover of the open dense subset,
$U_+TU_-$ of $G$. That is, they are defined on
$$
X:=U_+TU_-\underset{T}\x \mathfrak h. 
$$  
More generally we will be interested in subsets of $X$ of the form
$$
X_{\mathcal O}:=U_+ T U_-\underset{T}\x \mathcal O,
$$ 
where  $\mathcal O$ is a connected open subset in $\mathfrak h$ or $\mathfrak h_{\R}$.
We call a smooth function $f: X_{\mathcal O}\to \C$ a {\it Whittaker function} with respect to $\chi_+$ and 
$\chi_-$ if 
\begin{equation}\label{e:WhittakerCondition}
f(u_+e^h u_-,h)=e^{\chi_+(u_+)}f(e^h,h) e^{\chi_-(u_-)}.
\end{equation}
Clearly $f$ on $X_\mathcal O$ is completely determined by its restriction to the Cartan component, and 
conversely any smooth function on  $\mathcal O$ gives rise to a Whittaker function on $X_{\mathcal O}$.

Consider the representation of $\mathcal U(\mathfrak g)$ on $C^\infty(X_{\mathcal O})$ defined by
\begin{equation*}
\xi \cdot f (u_+ e^h u_-, h) :=  \left.\frac{d}{ds}\right |_{s=0} f(\exp(-s \xi)u_+ e^h u_-,h_s),\qquad \xi\in\mathfrak g_\R,
\end{equation*}  
where $s\mapsto h_s$ is the lift of the torus factor of $\exp(-s \xi)u_+ e^h u_-$
to $\mathfrak h_\R$ with $h_0=h$, defined for small enough $s$.

The connection between the quantum Toda Hamiltonian and Whittaker functions can now be stated as follows, compare \cite{Etingof:QWhittaker}.

\begin{enumerate}
\item
If $f$ is a Whittaker function (for $\chi_+$ and $\chi_-$) and $c$ is a central element in  $\mathcal U(\mathfrak g)$, then $c\cdot f$ is again a Whittaker function.
The resulting action of $\mathcal Z(\mathfrak g)$ on  Whittaker functions on $X_{\mathcal O}$ factors through a representation of $\mathcal Z(\mathfrak g)$ on  $C^{\infty}(\mathcal O)$, by restriction to the Cartan factor and extension.
\item 
Let $c_2$ be the degree two Casimir  such that $\gamma(c_2)(a)=\left <a,a\right >$. Then the action of $c_2$ on $C^{\infty}(\mathcal O)$
as defined in the previous equation together with (1) above is given by 
\begin{equation*}
c_2\cdot f = e^{\rho}\ \left(\frac 1 2\Delta + \sum_{i\in I}\chi_+^{(i)}\chi_-^{(i)}e^{\alpha_i}\right ) (e^{-\rho} f),\qquad f\in C^\infty(\mathcal O).
\end{equation*}
\end{enumerate}
Constructing solutions to the quantum Toda lattice is thereby equivalent to 
constructing  Whittaker functions annihilated by 
generators $\gamma\inv(\phi_i)$ of $\mathcal Z(\mathfrak g)$. 
[Explicitly, we will use $\chi_+^{(i)}=-\frac 1z$ and $\chi_-^{(i)}=\frac 1z$,
to find solutions for \eqref{e:QTodaHamiltonian}.] 
The construction of appropriate Whittaker functions 
is done via Kostant's theory of  `Whittaker modules'
and their (Whittaker vector) matrix coefficients \cite{Kos:qToda, Kos:Whittaker}. 

\subsection{Relation with the $A$-model of $G^\vee/B^\vee$} 
The quantum Toda lattice for $G$ appears in context of the the $A$-model of $G^\vee/B^\vee$ as
natural quantization of the (Givental-)Kim presentation for the quantum cohomology.
In that setting it can be referred to as the {\it quantum cohomology $D$-module}, or the {\it quantum 
differential equations} of $G^\vee/B^\vee$, see \cite{Giv:EquivGW,Kim:qCohGB}. 
Officially this means that it consists of the differential operators annihilating the 
components of Givental's $J$-function, 
$$
J(t_1,\dotsc, t_n)\in H^*(G^\vee/B^\vee,\C[t_1,\dotsc, t_n, z^{-1}])[[e^{t_1},\dotsc, e^{t_n}]],
$$
which is defined in terms of descendent 2-point Gromov-Witten invariants.
Therefore the Gromov-Witten theory of  
$G^\vee/B^\vee$ 
provides a full set of solutions 
inside 
$
\C[t_1,\dotsc, t_n, z^{-1}][[e^{t_1},\dotsc, e^{t_n}]]
$
to the quantum Toda lattice for $G$, namely
by expansion of the $J$-function with respect to the Schubert basis. 

A direct connection between Givental's $J$-function,
albeit in the $G$-equivariant case, 
and Whittaker modules (inside completed Verma modules) was 
made by A.~Braverman \cite{Braverman:InstantonCounting}.

\subsection{Relation with the $B$-model of $G^\vee/B^\vee$} 
A conjectural $B$-model to $G^\vee/B^\vee$ was constructed in \cite{Rie:MSgen}.
In this $B$-model there is a family (parameterized by $\mathfrak h$) of holomorphic
volume forms on $\mathcal R_{1,w_0}$, giving rise to `period integrals' varying 
over the family. The quantum Toda equations
are supposed to describe the variation of these periods as $h\in\mathfrak h$
varies.

For $SL_n$ the mirrors, written down in explicit coordinates by Givental \cite{Giv:MSFlag},
can be mapped to an open subset of $\mathcal R_{1,w_0}\x \mathfrak h$ in such a way that his $B$-model 
reappears as a pull-back of our Lie theoretic $B$-model, see \cite[Section~9]{Rie:MSgen}. A
connection between Givental's mirrors and Kostant's Whittaker modules in type $A$ was
worked out by GKLO \cite{GKLO:GaussGivental}.

Details of the general construction of a $B$-model for $G^\vee/B^\vee$ follow in Section~\ref{s:mirrors}. In preparation we recall the definition of our basic holomorphic
$N$-form on  $\mathcal R_{1,w_0}$ and a particular variant of it inspired by \cite{GKLO:GaussGivental}.

\section{Two regular $N$-forms on $\mathcal R_{1,w_0}$}\label{s:Nform}

The following is a special case of Proposition~7.2 in \cite{Rie:MSgen}.

\begin{prop}\label{p:omega} 
Let $\mathbf i=(i_1,\dotsc, i_N)\in I^N$ correspond to a reduced expression $s_{i_1}s_{i_2}\dotsc s_{i_N}$ of $w_0$ in $W$. 
There is a unique holomorphic $N$-form $\omega_{\mathbf i}$ on 
$\mathcal R_{1,w_0}$ 
such that the 
restriction of $\omega_{\mathbf i}$ to the open subset 
$$
\mathcal R_{\mathbf i}=\{x_{i_1}(a_1)\cdots x_{i_N}(a_N) B_-\ | \ a_i\in \C^*\}
$$ 
in $\mathcal R_{1,w_0}$ is given by
 $$
  \frac{d a_1}{a_1} \wedge \frac{d a_2}{a_2} \wedge\cdots \wedge \frac{da_N }{a_N}.
 $$ 
If  $\mathbf j$ is another reduced  expression of $w_0$, and is related to $\mathbf i$ by 
a single braid relation of length $m$, then 
$$
\omega_{\mathbf j}=(-1)^{m+1}\omega_{\mathbf i}.
$$
In particular the form $\omega_{\mathbf i}$ is independent of the reduced expression 
$\mathbf i$ up to sign. \qed
\end{prop}
We note that the subset $\mathcal R_{\mathbf i}$ in $\mathcal R_{1,w_0}$
is the open stratum in a stratification introduced by Deodhar~\cite{Deo:Decomp,MarRie:ansatz}. 

In the following we assume a reduced expression $\mathbf i$ has been chosen, and suppress the 
subscript $\mathbf i$, referring to the $N$-form defined in Proposition~\ref{p:omega} 
simply as $\omega$. 
The variant of our form $\omega$ defined by
\begin{equation}\label{e:wGKLO}
\omega_{GKLO}(uB_-):= \left<u\cdot v^-_{\rho }, v^+_{\rho }\right>
\omega(uB_-),
\end{equation}
 is a Lie-theoretic version of
 the $N$-form  introduced by \cite{GKLO:GaussGivental} in the 
type $A$ case. 

\begin{prop}\label{p:transformations} For any $g\in G$ denote by 
$\kappa_g:G/B_-\to G/B_-$ the map of 
left translation, $\kappa_{g}(g'B_-)=gg' B_-$. Let $\omega$ and $\omega_{GKLO}$ be
viewed as rational $N$-forms on $G/B_-$. Then we have the identities
\begin{align}\label{e:trans}
&\kappa_{g}^*\omega(uB_-)=\frac{\left< u\cdot v^-_\rho,v^+_\rho\right>}{\left<gu\cdot v^-_\rho,v^+_\rho\right>
\left<g u\cdot v^-_\rho,v^-_\rho\right>}\omega(uB_-), \\ \label{e:GKLOtrans}
&\kappa_{g}^*\omega_{GKLO}(uB_-)=\frac 1{\left<gu\cdot v^-_\rho,v^-_\rho\right>^2}\omega_{GKLO}(uB_-).
\end{align}
\end{prop}

\begin{rem}\label{r:transformations}
Proposition~\ref{p:transformations} implies the following identities.
\begin{enumerate}
 \item For $h\in\mathfrak h$ and Chevalley generators $f_i,e_i$ the volume form $\omega$
 transforms according to
 \begin{eqnarray*}
 \kappa^*_{\exp(h)}\omega&=&\omega,\\
 \kappa^*_{y_i(s)}\omega&=&\big (1+e_i^*(u)s\big)\inv\omega,\\
 \kappa^*_{x_i(s)}\omega&=&\left(1+ \frac{\left<u\cdot v^-_\rho,f_i\cdot v^+_\rho\right >}
 {\left<u\cdot v^-_\rho, v^+_\rho\right >}s\right)\inv\omega.\\
 \end{eqnarray*}
 \item The alternative volume form $\omega_{GKLO}$ is $U_+$-invariant. In particular it
 is well-defined on the entire big cell $U_+B_-/B_-$. However it is not $T$-invariant, satisfying
 instead
 $$
\kappa^*_{\exp(h)}\omega_{GKLO}=e^{2 \rho(h)}\omega_{GKLO}.
 $$ 
\end{enumerate}
\end{rem}
\begin{proof}
It is easy to see that the formulas \eqref{e:trans} and \eqref{e:GKLOtrans} 
are equivalent to one another. 
Now suppose $g_1,g_2\in G$ are such that the identity \eqref{e:GKLOtrans} holds. 
We claim that 
\begin{equation}\label{e:}
\kappa_{g_2}^*\left(\kappa_{g_1}^*\omega_{GKLO}\right)(uB_-)=
\frac 1{\left<g_1 g_2 u\cdot v^-_\rho,v^-_\rho\right>^2}\omega_{GKLO}(uB_-).
\end{equation}
If so, then the formula \eqref{e:GKLOtrans}  follows also for 
$g=g_1 g_2$, and therefore we need to check it (or equivalently \eqref{e:trans}) 
only on a generating subset of $G$.

To prove the claim compute
\begin{multline}\label{e:mult1}
\kappa_{g_2}^*\left(\kappa_{g_1}^*\omega_{GKLO}\right)(uB^-)=\kappa_{g_2}^*\left(
\frac 1{\left<g_1u\cdot v^-_\rho,v^-_\rho\right>^2}\omega_{GKLO}\right)(uB_-)=\\
\kappa_{g_2}^*\left(uB_-\mapsto
\frac 1{\left<g_1u\cdot v^-_\rho,v^-_\rho\right>^2}\right)
\frac 1{\left<g_2u\cdot v^-_\rho,v^-_\rho\right>^2} \omega_{GKLO}(uB_-).
\end{multline}
To apply $\kappa_{g_2}^*$ above we need a factorization
$$
g_2u=u_{g_2}b_{g_2}
$$
where $u_{g_2}\in U_+$ and $b_{g_2}\in B_-$. Then 
\begin{multline}\label{e:mult2}
\kappa_{g_2}^*\left(uB_-\mapsto
\frac 1{\left<g_1u\cdot v^-_\rho,v^-_\rho\right>^2}\right) (uB_-)=
\frac 1{\left<g_1u_{g_2}\cdot v^-_\rho,v^-_\rho\right>^2}\\
=
\frac 1{\left<g_1g_2 u b_{g_2}\inv \cdot v^-_\rho,v^-_\rho\right>^2}=
\frac 1{\left<g_1g_2 u \cdot v^-_\rho,v^-_\rho\right>^2} \frac 1{\rho (b_{g_2})^2}\\
=\frac 1{\left<g_1g_2 u \cdot v^-_\rho,v^-_\rho\right>^2} \left<g_2u\cdot v^-_\rho,v^-_\rho\right>^2.
\end{multline}
The claim now follows from the combination of \eqref{e:mult1} and \eqref{e:mult2}.

To finish the proof it suffices to check the identities from Remark~\ref{r:transformations}.
 Let us choose a reduced expression $\mathbf i$ of $w_0$ 
such that $\omega$ restricted to 
$$
\mathcal R_{\mathbf i}=\{x_{i_1}(a_1)\dotsc x_{i_N}(a_N)B_-\ |\ a_i\ne 0\}
$$ 
takes the form 
$$
\frac{da_1}{a_1}\wedge\cdots\wedge \frac{d a_N}{a_N}.
$$  
Note that it is clear from the definition that $\omega$ is invariant under translation by 
elements of $T$. In fact $\mathcal R_{\mathbf i}$ is itself a much bigger torus,
$\mathcal R_{\mathbf i} \cong (\C^*)^N$,  on which $\omega$ is 
the standard invariant $N$-form. We are therefore left with two kinds of transformations to consider,
$\kappa_{y_i(s)}$ and $\kappa_{x_i(s)}$.

\vskip .3cm
\noindent (1)\quad 
To work out the coordinate transformation corresponding to $\kappa_{y_i(s)}$
we note that $y_i(s)x_j(a)=x_j(a)y_i(s)$ for $i\ne j$ and 
$$
y_i(s)x_i(a)=x_i\left(\frac{a}{1+as}\right)\ \left(\frac{1}{1+as}\right)^{\hskip -.07cm\alpha_i^\vee}\
y_i\left(\frac{s}{1+as}\right).
$$
Supose $1\le l_1<\dotsc <l_m\le N$ are the indices for which $i_{l_j}=i$. Applying the above identities repeatedly we obtain
\begin{multline*}
y_i(s)x_{i_1}(a_1)\dotsc x_{i_N}(a_N)=
x_{i_1}(a_1)\dotsc x_{i_{l_1}}(a_{l_1}')\left(\frac {a_{l_1}'}{ a_{l_1}}\right)^{\hskip -.07cm\alpha_i^\vee} x_{i_{l_1+1}}(a_{l_1+1})\dotsc\\
\dotsc x_{i_{l_2}}(a_{l_2}')\left(\frac {a_{l_2}'} {a_{l_2}}\right)^{\hskip -.07cm\alpha_i^\vee}
\dotsc
x_{i_{l_m}}(a_{l_m}')\left(\frac {a_{l_m}'} {a_{l_m}}\right)^{\hskip -.07cm\alpha_i^\vee}\dotsc x_{i_N}(a_N)
y_i\left(\frac{s}{1+(a_{l_1}+\dotsc +a_{l_m})s}\right).
\end{multline*}
where
$$
a_{l_j}'=\frac{a_{l_j}(1+(a_{l_1}+\dotsc + a_{l_{j-1}})s)}{1+(a_{l_1}+\dotsc + a_{l_{j}})s}.
$$
Since $\omega$ restricted to $\mathcal R_{\mathbf i}$ is invariant under the
action of the `big' torus, that is $\mathcal R_{\mathbf i}$ itself, we may disregard
the factors  $\left(\frac {a_{l_j}'} {a_{l_j}}\right )^{\hskip -.07cm \alpha_i^\vee}$. Thus $\omega$ transforms
under $\kappa_{y_i(s)}$ as under the coordinate transformation 
$$
(a_1,\dotsc, a_N)\mapsto (a'_1,\dotsc, a'_N),
$$
where $a_j':=a_j$ if $j\notin\{l_1,\dotsc, l_m\}$. Since this coordinate transformation
is lower triangular its Jacobian is easily computed to be
$$
\det\left(\frac{\partial a'_j}{\partial a_k}\right)_{j,k}=\prod_{j=1}^m
\left(\frac{1+(a_{l_1}+\dotsc+a_{l_{j-1}})s}{1+(a_{l_1}+\dotsc+a_{l_{j}})s}\right)^2=
\left (\frac{1}{1+(a_{l_1}+\dotsc+a_{l_{m}})s}\right )^2.
$$
Note also  the telescopic product identity
$$
\frac{1} {\prod_{j=1}^N a'_j}=(1+(a_{l_1}+\dotsc+a_{l_{m}})s)\ \frac{1} {\prod_{j=1}^N a_j}.
$$
Therefore we obtain
\begin{multline*}
\frac{da'_1}{a'_1}\wedge\frac{da'_2}{a'_2}\wedge\cdots\wedge\frac{da'_N}{a'_N}=
 \frac{1}{\prod_{j=1}^N a'_j} \left(\frac{1}{1+(a_{l_1}+\dotsc+a_{l_{m}})s}\right)^2
da_1\wedge da_2\wedge\cdots\wedge da_N\\
=\frac{1}{1+(a_{l_1}+\dotsc+a_{l_{m}})s}\omega.
\end{multline*}
Clearly $a_{l_1}+\dotsc+a_{l_{m}}$ is nothing other than $e_i^*(u)$
for $u=x_{i_1}(a_1)\dotsc x_{i_N}(a_N)$, confirming the identity
from Remark~\ref{r:transformations}.

\vskip .3cm
\noindent (2)\quad 
To apply $\kappa_{x_i(s)}^*$ to $\omega$ let us assume without loss of generality that 
the reduced expression $\mathbf i$ of $w_0$ 
begins with $i_1=i$. 
Then $\kappa_{x_i(s)}$ corresponds to the coordinate transformation 
$$
(a_1,a_2,\dotsc, a_N)\mapsto (a_1+s,a_2,\dotsc, a_N),
$$
and therefore 
$$
\kappa_{x_i(s)}^*\omega= \frac{a_1}{a_1+s}\omega.
$$
It is easy to see that for $u=x_{i_1}(a_1)\dotsc x_{i_N}(a_N)$ we have indeed
$$
\left(1+\frac{\left<u\cdot v^-_\rho,f_i\cdot v^+_\rho\right >}
 {\left<u\cdot v^-_\rho, v^+_\rho\right >}s \right)\inv=
 \left(1+ \frac 1 {a_1}s\right)\inv=\frac{a_1}{a_1+s}.
$$
\end{proof}

\section{The mirror family to $G^\vee/B^\vee$}\label{s:mirrors}
In this section we will review the ingredients of mirror symmetry 
for the full flag variety $G^\vee/B^\vee$ following \cite{Rie:MSgen}. We 
will consider from the start the variant of the mirror family over 
$\mathfrak h$, rather than the family over $T$.  Also, for convenience, 
we have chosen $G$ simply connected, while in \cite{Rie:MSgen}
the group on the mirror symmetric side was adjoint. The mirror family over $\mathfrak h$
is however unaffected by this change. Finally, we
replace the mirror family from \cite{Rie:MSgen}
by its translate by $\dot w_0$, which will turn out to 
be more natural in connection with the Whittaker functions
we are constructing. 
\subsection{}\label{s:F}
Let
\begin{align}\label{e:Z}
Z&:=\{( g,h)\in U_+ T U_- \underset{T}\x \mathfrak h\ |\  g\cdot B_+=B_-\}.
\end{align}
$Z$ is viewed as a family of varieties via the map 
$\pr_2: Z\to \mathfrak h$
projecting onto the second factor. For $h\in \mathfrak h$ let us  write 
\begin{equation*}
Z_h:=  U_+ e^h U_-\cap B_-\dot w_0= \{g\in  U_+ e^h U_-\ |\ g\cdot B_+=B_-\}.
\end{equation*}
$Z_h$ may be identified with the fiber $\pr_2\inv(h)$ in $Z$. 
We record the following basic properties of  the family $Z$. 
\begin{enumerate}
\item
Fix $h\in \mathfrak h$. Then the fiber $Z_h$ is isomorphic to the intersection
of opposite big cells, $\mathcal R_{1,w_0}$, via the map
\begin{eqnarray*}
\beta_h:\ Z_h&\To &\mathcal R_{1,w_0}\\
\qquad\qquad ue^h\bar u\inv &\mapsto &u\inv\cdot B_-=e^{h}\bar u\inv\cdot B_+.
\end{eqnarray*}
In particular $Z_h$ is smooth of 
dimension $N=\dim_\C(G/B)$.
\item
The isomorphisms from (1) can be combined to 
give a trivialization
\begin{equation}\label{e:triv}
\beta:Z\overset\sim\To\mathcal R_{1,w_0}\x \mathfrak h,
\end{equation}
where $(u e^h\bar u\inv, h)\mapsto (u\inv\cdot B_-,h)$.
\end{enumerate}
Note that in particular
$$
Z_0= U_+ U_-\cap B_-\dot w_0=\{u \bar u\inv \in U_+ U_- \ | \ u\inv\cdot B_-=\bar u\inv\cdot B_+ \},
$$
and we have the isomorphism
\begin{equation}\label{e:beta0}
\beta_0: Z_0 \overset\sim\To \mathcal R_{1,w_0},\qquad u\bar u\inv\ \mapsto\  u\inv\cdot B_-=\bar u\inv\cdot B_+ .
\end{equation}
\subsection{}\label{s:superpotential}
We define a function $\mathcal F$ on $Z$ (the `superpotential') 
as follows.
\begin{equation}\label{e:F}
\mathcal F(u e^h\bar u\inv,h;z)=\frac 1 z\left( \sum_{i\in I} e_i^*(u )+\sum_{i\in I}f_i^*(\bar u)\right),
\end{equation}
where we may think of $z$ as a positive parameter $z\in\R_{>0}$. The restriction of
$\mathcal F$ to a fiber $Z_h$ of the mirror family (or rather $Z_h\x \R_{>0}$) is denoted by 
$\mathcal F_h$. 

\begin{rem} \label{r:Whittaker}Note that for fixed $z$ the function $e^{\mathcal F(\ ; z)}:Z\to \C$ is the restriction to $Z$ of  
the `trivial' Whittaker function, 
\begin{eqnarray*}
\mathcal W_0: U_+ T U_- \underset{T}\x \mathfrak h &\to & \C, \\
(u_+ e^h  u_-,h) &\mapsto & e^{\chi_+(u_+)} e^{\chi_-(u_-)},
\end{eqnarray*}
corresponding to characters $\chi_+,\chi_-$ defined by $\chi_+(e_i)=\frac 1z$ and
$\chi_-(f_i)=-\frac 1z$,  see also Section~\ref{s:Whittaker}.
\end{rem}

\subsection{}\label{s:omegah} Let $\mathbf i$ be a reduced expression for $w_0$ and  $\omega_{\mathbf i}$ 
the $N$-form on $\mathcal R_{1,w_0}$ defined in 
Proposition~\ref{p:omega}. Denote by 
$\omega_Z$ or $\omega_{\mathbf i, Z}$ the pullback of 
$\omega_\mathbf i$ to $Z$ along the map 
\begin{eqnarray*}
pr_1\circ \beta:& Z &\to \quad\mathcal R_{1,w_0},\\
\qquad &(u e^h\bar u\inv,h) &\mapsto\quad u\inv\cdot B_-,
\end{eqnarray*}
and write $ \omega_h$ or $\omega_{\mathbf i,h}$ for the pullback of
$\omega_Z$ to the fiber $Z_h$.

Note that the non-canonical choice of a reduced expression~$\mathbf i$
will affect at most the sign of $\omega_Z$. 
In some special cases we will consider later on, this sign will cancel out against a sign
coming from the choice of orientation for the integration cycle,
making for a canonical solution to the quantum Toda lattice.  

\vskip .3cm
The mirror datum to $G^\vee/B^\vee$ is now made up 
of the three ingredients introduced above: the family $Z\to\mathfrak h$,
the regular function $\mathcal F$, and the 
holomorphic $N$-form $\omega_Z$ on $Z$. 
We may denote it compactly as $(Z,\omega_Z, \mathcal F)$.

\subsection{A translation action on $Z$}\label{s:translation}
We can transfer the natural action of the additive group of $\mathfrak h$
on the product $\mathfrak h\x \mathcal R_{1,w_0}$, given by $h\cdot (h',g\cdot B_-):=(h'+h,g\cdot B_-)$, 
to $Z$ using the trivialization
$\beta$
from \eqref{e:triv}. 
\begin{lem}\label{l:translation}
Let $h\in \mathfrak h$ and $(h',g)\in Z$. The translation action of  $h$  
on $(h',g)$ takes the form
$$
h\cdot (h',g)=(h+h', g e^h).
$$
\end{lem}

\begin{proof}
We have that $g\in Z_{h'}$, hence we may write $g=u e^{h'}\bar u\inv$ for
$u\in U_+, \bar u\in U_-$. Then 
\begin{equation*}
ge^{h}= u e^{h'}\bar u\inv \ e^{h}=
u e^{h+h'}{(\bar u')}\inv,
\end{equation*}
where $\bar u'=e^{- h}\bar u e^{h}\in U_-$. Also $g e^h\cdot B_+=g\cdot B_+=B_-$. Hence 
$g e^h\in Z_{h+h'}$. Moreover  $\beta(h+h',g e^h)=(h+h', u\inv\cdot B_-)$
as required.
\end{proof}

\section{Distinguished cycles for integration}\label{s:contours}

\subsection{The compact cycles $\Gamma^{(1)}$}\label{s:gamma1}

\begin{defn}
Let $h\in\mathfrak h$ and let 
$\mathbf i=(i_1,\dotsc, i_N)$ be a reduced expression of $w_0$.
We may define $\Gamma_{\mathbf i,h}^{(1)}\subset Z_h$ to be
$$
\{u e^h \bar u\inv \in Z_h\ |\  u\inv\cdot B_-=x_{i_1}(a_1)\cdots x_{i_N}(a_N)\cdot B_-\ \text{where } \|a_j\|=1\text{ for all $j$ } \}.
$$
Note that  $\Gamma_{\mathbf i,h}^{(1)}$ is naturally isomorphic to a compact torus 
$(S^1)^N$. We define an associated $N$-cycle $[\Gamma_{\mathbf i,h}^{(1)}]\in H_N(Z_h,\Z)$ by choosing the 
anti-clockwise orientation on each $S^1$ factor. The resulting
family of integration contours is denoted $\Gamma^{(1)}_{\mathbf i}:=([\Gamma^{(1)}_{\mathbf i,h}])_{h\in \mathfrak h}$,
or just  $\Gamma^{(1)}$ if the choice of reduced expression is clear or irrelevant.
\end{defn}

\begin{lem}
The cycle $[\Gamma_{\mathbf i,h}^{(1)}]$ defines a nonzero element of $H_N(Z_h,\Z)$. Moreover, if 
two reduced expressions $\mathbf i$ and $\mathbf j$ are related by a braid relation of length $m$ then
$$
[\Gamma_{\mathbf i,h}^{(1)}]=(-1)^{m+1}[\Gamma_{\mathbf j,h}^{(1)}].
$$
In particular, the 
cycle $[\Gamma_{\mathbf i,h}^{(1)}]$ is independent of the reduced expression 
$\mathbf i$ up to sign. 
\end{lem}
\begin{proof}
We may identify $Z_h$ with $\mathcal R_{1,w_0}$ by $\beta_h$, and 
$\Gamma_{\mathbf i,h}^{(1)}$ with 
$$\Gamma_{\mathbf i}:=\{x_{i_1}(a_1)\cdots x_{i_N}(a_N)\cdot B_-\ |\  \|a_j\|=1\text{ for all $j$ }\},$$
and work in $\mathcal R_{1,w_0}$. By Section~\ref{s:Nform} we have a well defined $N$-form $\omega_{\mathbf i}$ on $\mathcal R_{1,w_0}$. It follows from
$$\int_{[\Gamma_{\mathbf i}]}\omega_{\mathbf i}=(2\pi i)^N
$$
that $[\Gamma_{\mathbf i}]$ defines a nontrivial homology class in $H_N(\mathcal R_{1,w_0},\Z)$.

Now $\Gamma_{\mathbf i}$ lies in the open coordinate patch $\mathcal R_{\mathbf i}$ (see Proposition~\ref{p:omega}) while
$\Gamma_{\mathbf j}$ lies inside $\mathcal R_{\mathbf j}$. Note that $\mathcal R_{\mathbf i}$ is 
isomorphic to $(\C^*)^N$ and therefore $[\Gamma_{\mathbf i}]$
generates $H_N(\mathcal R_{\mathbf i},\Z)$.
Rescaling the $S^1$ factors in $[\Gamma_{\mathbf j}]$ we may replace this 
cycle by one that is homologous but lies in the intersection 
$\mathcal R_{\mathbf i}\cap\mathcal R_{\mathbf j}$, and therefore
also defines a cycle $[\Gamma_{\mathbf j}']$ in $H_N(\mathcal R_{\mathbf i},\Z)$,
It follows that $[\Gamma_{\mathbf j}']$ must be a multiple of the (generating) cycle
$[\Gamma_{\mathbf i}]$ in $H_N(\mathcal R_{\mathbf i},\Z)$. Hence 
$[\Gamma_{\mathbf j}]$ is a multiple of $[\Gamma_{\mathbf i}]$
in $H_N(\mathcal R_{1,w_0},\Z)$. However we have by Proposition~\ref{p:omega} that 
$$
\omega_{\mathbf i}=(-1)^{m+1}\omega_{\mathbf j}.
$$
Therefore  
$$
\int_{[\Gamma_{\mathbf j}]}\omega_{\mathbf i}=(-1)^{m+1}
\int_{[\Gamma_{\mathbf j}]}\omega_{\mathbf j}=
(-1)^{m+1}(2\pi i)^N,
$$
which implies  $[\Gamma_{\mathbf j}]=(-1)^{m+1}[\Gamma_{\mathbf i}]$.
\end{proof}

\subsection{Total positivity and the non-compact integration cycles $\Gamma^{(w_0)}$}
\label{s:gammaw0}

\begin{defn}[Lusztig \cite{Lus:TotPos94}]
Inside the real points of a (split) real algebraic group $G$ and each of 
its related varieties $X=T, U_+, U_-$ and $G/B_-$, there is a semi-algebraic
subset $X^{>0}$ called the totally positive part, which is defined as follows. 

The totally positive part of $T$ is the precisely the subset
of $T$ for which all characters take values in $\R_{>0}$.
Equivalently, if we consider the real points of $T$ (isomorphic to $(\R^*)^n$),
then $T^{>0}$ is the connected component of the identity.

For $U_+$ and $U_-$ the totally positive parts are given by  
\begin{eqnarray*}
U_+^{>0}&:=&\{x_{i_1}(a_1)\dotsc x_{i_N}(a_N) \ | \ a_i\in \R_{>0}\},\\
U_-^{>0}&:=&\{y_{i_1}(a_1)\dotsc y_{i_N}(a_N) \ | \ a_i\in \R_{>0}\},
\end{eqnarray*}
where $\mathbf i=(i_1,\dotsc, i_N)$ is a (any) reduced expression
of $w_0$. One puts these together to build
\begin{eqnarray*}
G^{>0}&:=& U_+^{>0}T_{>0}U_-^{>0}=U_-^{>0}T_{>0}U_+^{>0},
\end{eqnarray*}
where a proof of the last identity may be found in \cite{Lus:TotPos94}.

The totally positive part of the flag variety $\mathcal B$ is
$$ 
\mathcal B^{>0}:=U_+^{>0}\cdot B_-=U_-^{>0}\cdot B_+.
$$
Again the last identity is proved in \cite{Lus:TotPos94}. 
Note that the totally positive part of $\mathcal B$ lies in $\mathcal R_{1,w_0}$.
We may also denote it by $\mathcal R_{1,w_0}^{>0}$. 
\end{defn}

These definitions generalize the classical notion of total positivity 
inside $GL_n$ developed by Polya and Schoenberg among others. 
The introduction of a theory of total positivity to flag varieties 
is due to Lusztig, even in type $A$.

\subsubsection{}
Let $h\in\mathfrak h_\R$, and note that $\exp(\mathfrak h_\R)=T_{>0}$. 
We can now use the trivialization $\beta$ to pull back the totally positive part
in $\mathcal R_{1,w_0}$ to the mirror fibers $Z_h$. We will refrain from 
calling this the totally positive part of $Z_h$, which instead will be 
defined differently in Section~\ref{s:pos}. 
Let 
\begin{multline*}
\Gamma^{(w_0)}_{h}:=\{g\in Z_h\ |\ \beta_h(g)\in \mathcal B^{>0}\}=\{ue^h\bar u\inv\in Z_h\ |\ u\inv\in U_+^{>0}\ \} \\
=\{ue^h\bar u\inv\in Z_h\ |\ \bar u\inv\in U_-^{>0}\ \},
\end{multline*}
where the final equality uses that $e^h\in T_{>0}$. We will see in 
Section~\ref{s:pos} that $\mathcal F_h$ always has a critical 
point in $\Gamma^{(w_0)}_{h}$.

For any choice of reduced expression $\mathbf i$ of $w_0$ we obtain 
a parameterization,
\begin{eqnarray*}
\R_{>0}^{N}&\overset\sim\to &\Gamma^{(w_0)}_{h}\\
(a_1,\dotsc, a_N)&\mapsto & \beta_h\inv(x_{i_1}(a_1)\dotsc x_{i_N}(a_N)\cdot B_-),
\end{eqnarray*}
which gives rise to an orientation on $\Gamma^{(w_0)}_{h}$. We denote by
$
[\Gamma^{(w_0)}_{\mathbf i,h}]
$
the oriented real (semi-algebraic) manifold inside $Z_h$ obtained in this way,
and by $\Gamma^{(w_0)}=\Gamma^{(w_0)}_{\mathbf i}=([\Gamma^{(w_0)}_{\mathbf i,h}])_{h\in\mathfrak h_{\R}}$,
the corresponding family over $\mathfrak h_{\R}$. 

By Proposition~\ref{p:omega},
if $\mathbf j$ is a reduced expression
obtained from $\mathbf i$ by a braid relation of length $m$, then 
the corresponding (rational subtraction-free) coordinate transformation 
$(a_1,\dotsc, a_N)\mapsto (a_1',\dotsc, a_N')$ reverses orientations
precisely if $m$ is even.
Therefore the orientation of $[\Gamma^{(w_0)}_{\mathbf i,h}]$
depends on $\mathbf i$ in the same way.

\subsection{Conditions on more general families}\label{s:gencontours}

Let $\mathcal O$ be a connected open subset of $\mathfrak h_\R$ or  $\mathfrak h$. Let 
$\Gamma=([\Gamma_h])_{h\in\mathcal O}$ be a continuous family
of  real, possibly non-compact,  $N$-dimensional
semianalytic sets $\Gamma_h$ in $Z_h$ with specified orientation, for which
$\operatorname{Re}(\mathcal F_h)\to -\infty$ 
in any non-compact direction of $\Gamma_h$, for a/any fixed $z>0$.
In this case 
\begin{equation}\label{e:integral}
\int_{[\Gamma_h]}e^{\mathcal F_{}}\omega_h
\end{equation}
is well-defined and absolutely convergent,
therefore differentiable in $h$.
This follows, as in a proof of F.~Pham~\cite[Appendix A5, Lemma (i)]{Pham:VanishingHom}, from estimates of Herrera 
\cite[A. Theorem 2.1(c)]{Herrera} on the norms of currents associated to locally closed semi-analytic 
subsets in real analytic manifolds.

\vskip .1cm
We will always assume that $0\in\mathcal O$, as we want $[\Gamma_h]$
to be equivalent 
to the translate $[h\cdot \Gamma_0]$
defined using Section~\ref{s:translation}, for small $h$.

We note that in certain cases  we can start with a  cycle $[\Gamma_0]$ in $Z_0$ 
and extend it automatically 
to a family over all of $\mathfrak h_\R$ with the above decay properties.
Namely, suppose $[\Gamma_0]$ has the property that in any 
non-compact direction of $\Gamma_0$
we have ${\operatorname {Re}}(\mathcal F_0)\to -\infty$, and all  
individual summands ${\operatorname{Re}}(f_i^*(\bar u))$ and  
${\operatorname {Re}}(e_i^*(u))$ are bounded from above.
This means that no summand of $\operatorname{Re}(\mathcal F_0)$ 
can tend to $+\infty$. In that case we can 
use the translation action  to define
$$
\Gamma_h:=h\cdot \Gamma_0=\{u \bar u\inv  e^h\in Z_h \  |\ u\bar u\inv\in \Gamma_0\}, 
$$
for $h\in \mathfrak h_{\R}$. The claim that $[\Gamma_h]$ again has the same decay behavior for 
$\operatorname{Re}(\mathcal F_h)$ as $[\Gamma_0]$ for $\operatorname{Re}(\mathcal F_0)$ follows 
 from the simple observation that
 $$
 \mathcal F_h(u\bar u\inv e^h)=\frac 1z\left(\sum_{i\in I} e_i^*(u)+\sum_{i\in I} f_i^*(e^{-h}\bar u e^h)\right)=
 \frac 1z\left(\sum_{i\in I} e_i^*(u)+\sum_{i\in I} e^{\alpha_i(h)}f_i^*(\bar u)\right), 
 $$
 and that $e^{\alpha_i(h)}>0$ (since $h\in\mathfrak h_\R$).

 For example the families of cycles $\Gamma^{(1)}$ and $\Gamma^{(w_0)}$ are obtained
 from $[\Gamma^{(1)}_{0}]$ and $[\Gamma^{(w_0)}_0]$ in this way.

\vskip .1cm

\begin{rem}
An alternative formal setting for integration cycles is 
the $N$-th rapid decay homology group associated 
to the irregular  rank one connection, $\nabla(f)=df- f\  d\mathcal F_h(\ ;z)$, 
defined
by $\mathcal F_{h}$ on the structure sheaf of $\mathcal R_{1,w_0}$ .
This homology group was 
defined by Bloch and Esnault
in dimension $1$ and more recently generalized by Hien~\cite{Hien:periods} 
to arbitrary dimension using work of T.~Mochizuki~\cite{Mo:Kashiwara}. 
The relevant to us case of exponential connections is treated in the earlier work 
of Hien and Roucairol~\cite{HienRoucairol:GaussManin}.

Another approach to constructing integration cycles, following Givental \cite{Giv:MSFlag}, is 
recalled in Section~\ref{s:pos}.
\end{rem}

\section{Statement of the main theorem}
\label{s:mainthm}

Let $\mathcal O$ be a connected open
subset  either of $\mathfrak h_\R$ or of $\mathfrak h$ that contains $0$.
Let $\Gamma=([\Gamma_h])_{h\in \mathcal O}$ be a family of real, possibly non-compact,
oriented $N$-dimensional semianalytic cycles  in $Z_h$, as described in 
Section \ref{s:gencontours}. 

We fix a reduced expression $\mathbf i$ of $w_0$ and let $\omega_h=\omega_{\mathbf i,h}$ 
be the $N$-form on $Z_h$ defined in Section~\ref{s:omegah}.  
Then let
\begin{equation}\label{e:Sagain}
S_{\Gamma}(h,z):=\int_{[\Gamma_{h}]}  
e^{\mathcal F}\ \omega_h,
\end{equation}
for $h\in \mathcal O$ and $z\in\R_{>0}$. 
For example 
\begin{equation}\label{e:S+-}
S_{\Gamma^{(1)}}(h,z)=\int_{[\Gamma_{\mathbf i, h}^{(1)}]}  
e^{\mathcal F}\ \omega_{\mathbf i, h} \quad \text{and}\quad 
S_{\Gamma^{(w_0)}}(h,z)=\int_{[\Gamma^{(w_0)}_{\mathbf i, h}]}  
e^{\mathcal F}\ \omega_{\mathbf i, h},
\end{equation}
using the integration cycles defined in Section~\ref{s:contours}.

In general, the sign of $S_{\Gamma}(h,z)$ depends
on the reduced expression $\mathbf i$ used to define
$\omega_h$. However the special solutions $S_{\Gamma^{(1)}}$ 
and $S_{\Gamma^{(w_0)}}$, 
where we have chosen the orientation of
the integration cycle concurrently, are independent of $\mathbf i$. 
Note that $S_{\Gamma^{(1)}}(h,z)$ extends to a global holomorphic function on
 $\mathfrak h_\C\x \C^*$.

The following result was conjectured in \cite{Rie:MSgen}.
 
\begin{thm}\label{t:main}
The integrals \eqref{e:Sagain} are solutions to the 
quantum Toda lattice. In particular they are annihilated
by the quantum Toda Hamiltonian \eqref{e:QTodaHamiltonian}.
\end{thm}

\section{A $\mathcal U(\mathfrak g)$-module structure on $\mathcal Hol(Z_0)$.}\label{s:modulestructure}

We consider the restriction of the complex 
line bundle $L_{-\rho}=G\x_{B_-}\C_{-\rho}$ to the intersection of
opposite big cells $\mathcal R_{1,w_0}$. 
Since $\mathcal
R_{1,w_0}$ is open in $G/B_-$, the representation of $G$
on the space of sections 
induces a representation of 
$\mathfrak g$ on $\Gamma_{hol}(L_{-\rho}|_{\mathcal R_{1,w_0}})$,
the space of holomorphic sections of the restricted line bundle.
Moreover, since $\mathcal R_{1,w_0}$ is preserved by $T$,
we have a corresponding representation of 
$T$ on 
$\Gamma_{hol}(L_{-\rho}|_{\mathcal R_{1,w_0}})$ which is compatible with
the $\mathfrak g$-module structure.
Explicitly, let us set
\begin{multline*}
M_{-\rho}:=\Gamma_{hol}(L_{-\rho}|_{\mathcal R_{1,w_0}})\\=
\{\tilde f~: (U_+\cap B_-\dot w_0 B_-)B_-\to \C \ | \ \tilde f\ {\rm holomorphic\ }, \tilde f(gb)=\tilde f(g)\rho(b),\ \forall b\in B_- \}.
\end{multline*}
The actions of $\mathfrak g$ and $T$ on $M_{-\rho}$ are given by 
\begin{eqnarray}
(X\cdot \tilde f)(g)&:=&\left .\frac{d}{ds}\right|_{s=0} \tilde f(\exp(-sX)g),\\ 
t\cdot \tilde f(g)&:=&\tilde f(t\inv g), 
\end{eqnarray}
for $X\in \mathfrak g$ and $t\in T$. Note that the $\mathcal U(\mathfrak g)$-module
$M_{-\rho}$ has zero infinitesimal character, see \cite[Proposition~5.1]{Kos:Whittaker}.

The restriction of $\tilde f\in M_{-\rho}$ to $U_+\cap B_-\dot w_0 B_-$ defines an 
isomorphism,
\begin{eqnarray*}\label{e:MtoR}
M_{-\rho}&\overset \sim\To &\mathcal Hol(\mathcal R_{1,w_0}),\\
\tilde f\ \ & \mapsto &\left (f: u\cdot B_-\mapsto \tilde f(u)\right), \qquad\text{where }\ u\in U_+\cap B_-\dot w_0 B_-,
\end{eqnarray*}  
identifying $M_{-\rho}$ with holomorphic functions on $\mathcal R_{1,w_0}$. The actions of $\mathfrak g$ and $T$ on $M_{-\rho}$ thereby carry over to representations on $\mathcal Hol(\mathcal R_{1,w_0})$.

Consider now the zero fiber, $Z_0$, of our mirror family. We obtain a $\mathfrak g$-module 
and compatible $T$-module structure on $\mathcal Hol(Z_0)$ 
via the isomorphism $\beta_0: Z_0 \To \mathcal  R_{1,w_0}$ from \eqref{e:beta0}.
By construction, this representation of $\mathfrak g$ on $\mathcal Hol(Z_0)$ extends
to a representation of $\mathcal U(\mathfrak g)$ with zero infinitesimal character.

\subsection{A $\mathfrak u_+$-Whittaker vector}
Let $\chi:\mathfrak u_+\to\C$ be the $1$-dimensional representation 
defined by $\chi(e_i)=\frac 1 z$, for all $i\in I$, 
and consider the corresponding holomorphic character $e^{\chi}$ on $U_+$. 
Let $\psi_+\in \mathcal Hol(Z_0)$ be defined by 
$$
\psi_+(u \bar u\inv):= e^{\chi(u)},
$$
and let $\widetilde{\psi}_+\in M_{-\rho}$
denote the section  of $L_{-\rho}|_{\mathcal R_{1,w_0}}$ associated to $\psi_+$.
\begin{lem}
$\psi_+$ is a $\mathfrak u_+$-Whittaker vector in $\mathcal Hol(Z_0)$ with character $\chi$. That is,
$$
e_i\cdot \psi_+=\frac 1 z \psi_+, \qquad\text{ for all $i\in I$.}
$$ 
\end{lem}
\begin{proof}
From the definitions we see 
\begin{multline*}
(e_j\cdot \psi_+)(u \bar u\inv)=(e_j\cdot\widetilde{ \psi}_+)(u\inv)= \left .\frac d{ds}\right |_{s=0} \widetilde{\psi}_+(\exp(-s e_j) u\inv)\\=
\left .\frac d{ds}\right |_{s=0}\exp\left(\frac 1 z \sum_{i\in I} e_i^*(u \exp(s e_j))\right) =
\left .\frac d{ds}\right |_{s=0}\exp\left(\frac 1 z\left (s+\sum_{i\in I} e_i^*(u)\right)\right)\\
=\frac 1 z \psi_+(u\bar u\inv).
\end{multline*}
\end{proof}

\begin{defn}[\cite{Kos:Whittaker}]
A $\mathcal U(\mathfrak g)$-module with a cyclic Whittaker vector is
called a {\it Whittaker module}. 
\end{defn}

\begin{defn}
Let $V_+$ be the $\mathcal U(\mathfrak g)$-submodule
of $\mathcal Hol(Z_0)$ generated by $\psi_+$.
\end{defn}

By   \cite[Theorem 3.6.2]{Kos:Whittaker}  $V_+$ is an irreducible 
Whittaker module, 
see also \cite{SoergelMilicic:Whittaker}. Note that $V_+$ no longer
has an action of $T$.

\subsection{A $\mathfrak u_-$-Whittaker vector}

Let $\bar\chi:\mathfrak u_-\to\C$ be the $1$-dimensional representation 
defined by $\bar\chi(f_i)=\frac 1 z$, for all $i\in I$, 
and consider the corresponding holomorphic character $e^{\bar\chi}$ on $U_-$. 
Let $\psi_-\in\mathcal Hol(Z_0)$ be defined by 
$$
\psi_-(u\bar u\inv):=\frac{1}{\left<u\inv\cdot v_{\rho}^-,v_{\rho}^+\right>} e^{\bar\chi(\bar u)}.
$$
The functions $\psi_+$ and $\psi_-$ are  Lie-theoretic analogues of the functions introduced in terms of
Givental coordinates in \cite{GKLO:GaussGivental}. 

\begin{lem}\label{p:Whittaker} $\psi_-\in\mathcal Hol(Z_0)$  is a  $\mathfrak u_-$-Whittaker
vector  with character $\bar \chi$, That is,
$$
f_i\cdot\psi_-=\frac 1 z \psi_-.
$$
\end{lem}

In the following lemma we collect some identities used in the proof of the 
Lemma~\ref{p:Whittaker}.

\begin{lem}\label{l:yi} Suppose $u\in U_+,\bar u\in U_- $ are  given such that
$u\bar u\inv$ lies in $Z_0$. Consider a fixed 
Chevalley generator $f_{i}$.
\begin{enumerate}
\item For $s\in \C$ such that $1+s e_{i}^*(u)\ne 0$ we have the identity
$$
u y_{i}(s)=b_{(s)}u_{(s)},
$$
where $b_{(s)}\in B_-$ and $u_{(s)}\in U_+$ are given explicitly by  
\begin{eqnarray*}
b_{(s)}&=&\left ({1+s e_{i}^*(u)}\right)^{\alpha_{i}^\vee} y_{i}\big (s(1+s e_{i}^*(u))\big)\\
u_{(s)}&=&y_{i}\big (-s(1+s e_{i}^*(u))\big)\, \left (\frac 1{1+s e_{i}^*(u)}\right)^{\hskip -.07cm\alpha_{i}^\vee} u\,  y_{i}(s).
\end{eqnarray*}
\item
Let $s$ be as above, and define $\bar u_{(s)}$ by 
$$
\bar u_{(s)}\inv\cdot B^+=u_{(s)}\inv\cdot B^-.
$$
Then $\bar u_{(s)}=\bar u\ y_{i}(s)$.
The element $u_{(s)}\bar u_{(s)}\inv$  in $Z_0$ is given by
$$
y_{i}\big (-s(1+s e_{i}^*(u))\big)\, \left (\frac 1{1+s e_{i}^*(u)}\right)^{\hskip -.07cm\alpha_{i}^\vee}  u\bar u\inv.
$$
\end{enumerate}
\end{lem}
The proof of the Lemma is straightforward.

\begin{proof} [Proof of Lemma~\ref{p:Whittaker}]
To analyze the action of $f_{i}$ on $\psi_-$ we use Lemma~\ref{l:yi} with all its notations.
Let $\tilde\psi_-$ denote the element of $M_{-\rho}$
associated to $\psi_-$. 
\begin{multline*}
(f_i\cdot \psi_-) (u\bar u\inv)=( f_i\cdot \widetilde{\psi}_-) (u\inv)=\left .\frac d{ds}\right |_{s=0}
\widetilde{\psi}_-(\exp(-s f_i)u\inv)\\
=\left .\frac d{ds}\right |_{s=0}
\widetilde{\psi}_-(u_{(s)}\inv b_{(s)}\inv)=
\left .\frac d{ds}\right |_{s=0}
\widetilde{\psi}_-(u_{(s)}\inv )\rho(b_{(s)})\inv
=\left .\frac d{ds}\right |_{s=0}
\widetilde{\psi}_-(u_{(s)}\inv )\frac{1}{(1+s e_i^*(u))}.
\end{multline*}
Now
\begin{equation*}
\widetilde{\psi}_-(u_{(s)}\inv )={\psi}_-(u_{(s)} \bar u_{(s)}\inv)=\frac 1{\left<u_{(s)}\inv\cdot v_{\rho}^-,v_{\rho}^+\right>}
e^{\bar\chi (\bar u_{(s)})}=\frac{(1+s e_i^*(u))}{\left<u\inv\cdot v_{\rho}^-,v_{\rho}^+\right>} e^{\bar\chi (\bar u)+\frac 1 z s},
\end{equation*}
using the formulas for $u_{(s)}$ and $\bar u_{s}$ from Lemma~\ref{l:yi}. Therefore we get
$$
(f_i\cdot \psi_-)(u\bar u\inv)=
\left .\frac d{ds}\right |_{s=0}\frac 1{\left<u\inv\cdot v_{\rho}^-,v_{\rho}^+\right>} 
e^{\bar\chi (\bar u)+\frac 1 z s}=\frac 1z\psi_-(u\bar u\inv).
$$
\end{proof}

\begin{defn}
Let $V_-$ be the $\mathcal U(\mathfrak g)$ module in 
$\mathcal Hol(Z_0)$ generated by  $\psi_-$. 
This is another irreducible Whittaker module.
\end{defn}

\section{A Whittaker functional on $V_+$}
Consider again the holomorphic character $e^{\bar\chi}$ on $U_-$ given by
$$
\bar u\mapsto e^{\bar\chi(\bar u)}=e^{\frac 1z\sum f_i^*(\bar u)}.
$$
Suppose we have fixed a cycle $[\Gamma_0]$ in $Z_0$ as in Section~\ref{s:gencontours}. 
Then we define a linear map $\Psi_-^{[\Gamma_0]}:V_+\to \C$ by 
$$
\Psi^{[\Gamma_0]}_-(f):=\int_{[\Gamma_0]}e^{\bar\chi(\bar u)}\, f\ \omega.
$$ 
 \begin{rem}\label{r:convergence}
Recall how the integration cycles were chosen in Section~\ref{s:gencontours} 
for $e^\mathcal F$ to have exponential decay in any non-compact direction. 
 Now $\Psi_-^{[\Gamma_0]}(\psi_+)=\int_{[\Gamma_0]} e^{\bar \chi(\bar u)}\psi_+\ \omega=
 \int_{[\Gamma_0]} e^{\mathcal F}\ \omega$. 
Since repeated actions by 
elements of $\mathfrak g$ on the $\psi_+$ produce only rational amplitude factors
which do not affect convergence, $\Psi^{[\Gamma_0]}_-$ is defined on all of $V_+$. 
 \end{rem} 

Let us consider $\Psi^{[\Gamma_0]}_-$ as an element of the 
$\mathcal U(\mathfrak g)$-module $V_+^*$ dual to $V_+$.

\begin{prop} \label{t:Whittaker} 
For all $i\in I$, we have
$$
f_i\cdot \Psi_-=\frac 1z \Psi_-.
$$
That is, $\Psi_-$ is a $\mathfrak u_-$-Whittaker vector in $V_+^*$.
\end{prop}

\subsection{A bilinear pairing}
To prove Proposition~\ref{t:Whittaker} we will construct a pairing between the 
Whittaker modules $V_+$ and $V_-$, such that 
the Whittaker vector $\psi_-$ becomes identified
with the Whittaker functional $\Psi_-$ from Theorem~\ref{t:Whittaker}. 
This pairing essentially generalizes one introduced in \cite{GKLO:GaussGivental}.

Let $[\Gamma_0]$ be
a middle-dimensional cycle in $Z_0$ as above. 
Consider integrals of the form
$$
\int_{[\Gamma_0]}\phi\ \psi\  \  \omega_{GKLO},
 $$
 where $\phi,\psi$ are holomorphic functions
 on $Z_0$ and $\omega_{GKLO}$
 is the volume form from \eqref{e:wGKLO}, pulled back
 to $Z_0$ via the isomorphism $\beta_0:Z_0\to\mathcal R_{1,w_0}$.

\begin{prop}\label{p:pairing}\label{c:invariance}
Let $\phi\in V_-$ and $\psi\in V_+$.
\begin{enumerate}
\item The formula
$$
\left<\phi,\psi\right>_{[\Gamma_0]}:=\int_{[\Gamma_0]}\phi\ \psi \  \omega_{GKLO}
$$
 defines a bilinear pairing, $\left<\phi,\psi\right>_{[\Gamma_0]}:V_-\x V_+\to \C$. 
\item For any $X\in\mathfrak g$,
\begin{equation*}
\left< X\cdot \phi,\psi\right>_{[\Gamma_0]} + \left<\phi, X\cdot \psi\right>_{[\Gamma_0]}=0.
\end{equation*}
\end{enumerate}
\end{prop}

\begin{proof}
The pairing $\left<\ ,\ \right>_{[\Gamma_0]}$ is well-defined for the same reason as
$\Psi_-^{[\Gamma_0]}$, see Remark~\ref{r:convergence}. 
Let us now prove (2). We transfer the integrals from $Z_0$ to $\mathcal R_{1,w_0}$ via $\beta_0$,
but keeping the notation the same. For a 
function $f$ on $\mathcal R_{1,w_0}$ identified with a function on $x\in U_+\cap B_-\dot w_0 B_-$, 
we denote by $\tilde f$ the corresponding element of $M_{-\rho}$ given by 
$$
\tilde f(x b_-)=f(x)\rho(b_-), \quad {\text{for $b_-\in B_-$,}}
$$
and by $\bar f$ the usual function 
$$
\bar f(x b_-)=f(x), \quad {\text{for $b_-\in B_-$}}
$$
on $(U_+\cap B_-\dot w_0 B_-)B_-$ given by $f$.

The proposition relies on the observation that 
$\omega_{GKLO}$
`compensates' for the weight $-\rho$ 
twists coming from the representation on $M_{-\rho}$. Namely, on the 
level of $N$-forms we claim that
\begin{equation}\label{e:twist}
 (\exp(sX)\cdot \phi ) (\exp(sX)\cdot \psi )\ \omega_{GKLO}=\kappa^*_{\exp(-sX)}( \phi\ \psi\ \omega_{GKLO}),
\end{equation}
where on the left hand side we have the local action of $\exp(sX)$ on 
sections of $L_{-\rho}$, and on the right hand side we have the pull-back of forms, as
in Proposition~\ref{p:transformations}.

To prove \eqref{e:twist},  write 
$\exp(-sX) x=x_{(s)} b_-$ for $x_{(s)}\in U_+$ and $b_-\in B_-$. Then 
$$
(\exp(sX)\cdot \phi ) (\exp(sX)\cdot \psi) = \tilde\phi(x_{(s)}b_-) \tilde\psi(x_{(s)}b_-)=
 \phi(x_{(s)}) \psi(x_{(s)}) \rho(b_-)^2.
$$
On the other hand we have 
\begin{multline*}
\kappa^*_{\exp(-sX)}( \phi\ \psi\ \omega_{GKLO})(x)=\bar \phi(\exp(-sX)x)\bar\psi(\exp(-sX)x)\kappa^*_{\exp(-sX)}\omega_{GKLO}
\\=
\bar \phi(\exp(-sX)x)\bar\psi(\exp(-sX)x)\frac 1{\left<\exp(-sX)x\cdot v_{-\rho},v_{-\rho}\right>^2}\omega_{GKLO}
\\=\phi(x_{(s)})\psi(x_{(s)})\frac 1{\left< x_{(s)}b_-\cdot v_{-\rho},v_{-\rho}\right>^2}\omega_{GKLO}=
\phi(x_{(s)})\psi(x_{(s)}){\rho(b_-)^2}\omega_{GKLO},
\end{multline*}
using Proposition~\ref{p:transformations} and that $x_{(s)}\in U_+$.
This completes the
proof of the identity~\eqref{e:twist}.

Taking the derivative of~\eqref{e:twist} we therefore have
\begin{multline*}\label{e:der}
(\, (X\cdot \phi )\  \psi +  \phi \ (X\cdot \psi )\, )\ \omega_{GKLO} =
\left.\frac{d}{ds}\right|_{s=0} (\exp(sX)\cdot \phi ) (\exp(sX)\cdot \psi )\ \omega_{GKLO}\\
=\left.\frac{d}{ds}\right|_{s=0} \kappa^*_{\exp(-sX)}( \phi\ \psi\ \omega_{GKLO})=
d\circ i_{-X}( \phi\ \psi\ \omega_{GKLO})
\end{multline*}
where in the last expression $X$ is understood as the vector field defined by the action of $X$. 

The corresponding integral
$$
\int_{[\Gamma_0]} ((X\cdot \phi )\  \psi \ + \phi \ (X\cdot \psi ) )\ \omega_{GKLO}=
\int_{[\Gamma_0]} d i_{-X}( \phi\ \psi\ \omega_{GKLO})
$$
clearly vanishes if $\Gamma_0$ is compact, by the usual Stokes' theorem. In general, 
one can see that the right hand side vanishes using the arguments of Hien and Roucairol
\cite{HienRoucairol:GaussManin}, by replacing $\Gamma_0$ 
by its closure in a suitable compactification of $Z_0$
and applying the the `limit Stokes formula' \cite[Section~2.1]{HienRoucairol:GaussManin}.
The identity in (2) follows.
\end{proof}

\begin{proof}[Proof of Proposition~\ref{t:Whittaker}]
For any $\psi\in V_+$ we have
 \begin{multline*}
f_i\cdot \Psi^{[\Gamma_0]}_-(\psi)=\left<\psi_-, -f_i\cdot\psi\right>_{[\Gamma_0]}
=\left<f_i\cdot \psi_-, \psi\right>_{[\Gamma_0]}=
\frac{1}z \left<\psi_-, \psi\right>_{[\Gamma_0]}=\frac 1 z \Psi^{[\Gamma_0]}_-(\psi),
\end{multline*}
by Proposition~\ref{c:invariance} and then Lemma~\ref{p:Whittaker}.
\end{proof}

\begin{proof}[Proof of Theorem~\ref{t:main}]
For $g\in G$ consider the `matrix coefficient' 
\begin{equation}\label{e:matrixcoeff}
g\mapsto \Psi^{\Gamma}_-(g\inv\cdot \psi_+).
\end{equation}
Strictly speaking only $\mathcal U(\mathfrak g)$ acts on $V_+$, so the above definition doesn't make 
sense for general $g\in G$. However, for  $[\Gamma_0]$ a member of a 
family of cycles $[\Gamma_h]_{h\in\mathcal O}$ as in Section~\ref{s:gencontours}, 
 this particular matrix coefficient turns out to be 
 well defined on the corresponding $X_{\mathcal O}$ 
 (compare Section~\ref{s:Whittaker}), giving
\begin{eqnarray*}
X_{\mathcal O}:=U_+TU_-\underset{T}\x \mathcal O & \to & \C\\
(u_+e^h u_-,h)&\mapsto &  \Psi^{\Gamma}_-(u_-\inv e^{-h} u_+\inv\cdot \psi_+)=
e^{\chi_-(u_-)}\Psi^{\Gamma}_-(e^{-h}\cdot \psi_+) e^{\chi_+(u_+)},
\end{eqnarray*}
where $\chi^{(i)}_+=-\frac 1 z$, $\chi_-^{(i)}=\frac 1z$, for all $i$, and 
we compute $\Psi^{\Gamma}_-(e^{-h}\cdot \psi_+)$ below. 
In this way \eqref{e:matrixcoeff} gives rise to a Whittaker function, which we denote
by $\mathcal W_\Gamma$.

Since $V_+$, and with it $V_+^*$, have zero infinitesimal character, $\mathcal W_{\Gamma}$  is  annihilated by the Casimir generators in $\mathcal Z(\mathfrak g)$.
By Section~\ref{s:Whittaker}~(2), therefore,
$$
e^{-\rho}\ \mathcal W_\Gamma |_{\mathcal O}\quad : \quad h\mapsto e^{-\rho(h)}\ \Psi_-^{\Gamma}(e^{-h}\cdot \psi_+)
$$
is a solution to the 
quantum Toda lattice \eqref{e:QTodaHamiltonian}.

To compute this solution, suppose $h\in \mathcal O$. Then
\begin{multline}\label{p:intidentity}
e^{-\rho(h)}\  \Psi^{\Gamma}_-(e^{-h}\cdot \psi_+)=e^{-\rho(h)} \int_{[\Gamma_0]} 
e^{\bar\chi(\bar u)} (e^{-h}\cdot \psi_+)(u\bar u\inv)\omega_0\\
=e^{-\rho(h)} \int_{[\Gamma_0]} e^{\bar\chi(\bar u)} (e^{-h}\cdot\tilde{ \psi}_+)(u\inv)\omega_0
=e^{-\rho(h)} \int_{[\Gamma_0]} e^{\bar\chi(\bar u)} \ \widetilde{ \psi}_+(e^h u\inv)\omega_0\\
=
e^{-\rho(h)} \int_{[\Gamma_0]} e^{\bar\chi(\bar u)} \ \widetilde{ \psi}_+ (e^h u\inv e^{-h})\rho(e^h)\omega_0
=\int_{[\Gamma_0]} e^{\bar\chi(\bar u)}\ \widetilde{ \psi}_+(e^h u\inv e^{-h})\omega_0\\
=\int_{[\Gamma_0]} e^{\bar\chi(\bar u)}\ \psi_+(e^h u e^{-h})\omega_0=
\int_{[\Gamma_0]} e^{\bar\chi(\bar u)}\ e^{\chi (e^h u e^{-h})}\omega_0.
\end{multline}
Now we can use Proposition~\ref{p:pairing} (2) to rewrite the integral
$$
\int_{[\Gamma_0]} e^{\bar\chi(\bar u)}\ e^{\chi (e^h u e^{-h})}\omega_0=
\int_{[\Gamma_0]} e^{\chi (u)}\ e^{\bar\chi( e^{-h}\bar u e^{h})}\omega_0,
$$
to get precisely 
\begin{multline}
e^{-\rho(h)}\  \Psi_-(e^{-h}\cdot \psi_+)= 
\int_{[\Gamma_h]} e^{\chi(u)+\bar\chi(\bar u) }\ \omega_h = 
\int_{[\Gamma_h]} e^{\frac 1 z(\sum e_i^*(u)+\sum f_i^*(\bar u)) }\ \omega_h,
\end{multline}
where $[\Gamma_h]$ is the translate of $[\Gamma_0]$, as in Section~\ref{s:gencontours}.
This completes the proof of the theorem.
\end{proof}

\section{Total positivity and critical points of $\mathcal F_h$}\label{s:pos}

Let us fix $z>0$. By the
mirror symmetric construction of the quantum cohomology
ring of $G^\vee/B^\vee$ proved in \cite{Rie:MSgen},
the critical points of the $\mathcal F_h=\mathcal F_h(\  ; z)$ (for
varying $h$) sweep out the Peterson variety 
$$
Y_B^*=Y\underset{G/B}\x\mathcal R_{1,w_0},
$$
with $h$ determining the values of the quantum
parameters. Explicitly,
$g\in Z_h$ is a critical point of $\mathcal F_h$ 
precisely if $g\cdot B_-\in Y_{B}^*$ with $q_i(g\cdot B_-)=e^{\alpha_i(h)}$. 
On the other hand by \cite{Kos:QCoh}  the quantum cohomology ring of 
the full flag variety $G^\vee/B^\vee$ is semisimple
for generic choice of quantum 
parameters. This implies
that  
$\mathcal F_h$ has precisely $\dim H^*(G^\vee/B^\vee)=|W|$
 critical points, all non-degenerate, for generic $h$.

Following Givental \cite{Giv:MSFlag}, the critical points are directly 
related to 
integration cycles. Namely to any non-degenerate
critical  point $p$ in $Z_h$, Givental associates a 
 `descending gradient cycle' for $\operatorname{Re}(\mathcal F)$.
In this way, one may obtain $|W|$
cycles in a generic mirror fiber $Z_h$, which 
should provide a basis of 
solutions to the quantum Toda lattice.

\begin{defn}\label{d:totneg}
Consider the isomorphism $\delta_h:Z_h\to \mathcal R_{1,w_0}$ given by $g\mapsto g\cdot B_-$. Let 
$Z_h^{>0}:=\delta_h\inv(\mathcal R_{1,w_0}^{>0})$.  
\end{defn}

\begin{lem}\label{l:pos}
Suppose $u\bar u\inv\in Z_0^{>0}$. Then  
$u\in U_+^{>0}$ and $\bar u\in U_-^{>0}$.   
\end{lem}
\begin{proof}
It is clear from the definitions that $u\in U_+^{>0}$. The rest of the 
lemma follows from $u\inv\cdot B_-=\bar u\inv\cdot B_+$,  together
with Lusztig's result \cite{Lus:TotPos94} that
$$
U_+^{>0}\cdot B_-= U_-^{>0}\cdot B_+,
$$
 applied to the 
opposite pinning, where $\bar x_i(t)=\exp(-te_i)$ 
and $\bar y_i(t)=\exp(-tf_i)$.
\end{proof}

\begin{prop}\label{p:poscrit}
For any $h\in\mathfrak h_\R$ and $M\in \R_{>0}$, consider the set 
$$
\mathcal M_{h,M}:=\left \{ g\in Z_h^{>0} \ |\  \mathcal F(g;z)\le \frac 1 z M\right \}.
$$
If $M$ is sufficiently large then $\mathcal M_{h,M}$ is a nonempty,
compact subset of $Z_h^{>0}$.  In particular the restriction of $\mathcal F_h$ 
to $Z_h^{>0}$ attains a minimum.
\end{prop}
Note that by definition $\mathcal M_{h,M}$ is independent of the positive scalar
$z$.
\begin{proof}
If $ue^h\bar u\inv\in \mathcal M_{h,M}$, then by Lemma~\ref{l:pos} 
we have $u\in U_+^{>0}$ and $\bar u\in U_-^{>0}$. So we can 
fix a reduced expression $\mathbf i$ of $w_0$ and write
\begin{eqnarray*}
u&=&x_{i_1}(a_1)x_{i_2}(a_2)\cdots x_{i_N}(a_N),\\
\bar u&=&y_{i_1}(b_1)y_{i_2}(b_2)\cdots y_{i_N}(b_N),
\end{eqnarray*}
for positive $a_i,b_i$ with $i\in I$. Using the $a_i$ and $b_i$
as coordinates for $u$ and $\bar u$, respectively, we
may define 
$$
\mathcal N_{h,M}:=\{ u e^h\bar u\inv\in Z_h^{>0} \ |\  
a_i\le M, \ b_i\le M \quad\forall i\in I\,  \}.
$$
 Since $\mathcal F(u e^h \bar u\inv;z)=\frac {1}z\left(\sum a_i+\sum b_i\right)$, 
it is clear that $\mathcal M_{h,M}$ is a closed subset of $ \mathcal N_{h,M}$.
It suffices therefore to show the following claim.
\vskip .3cm
\noindent {\it Claim:} There exists an $m<M$ such that 
$\mathcal N_{h,M}$ is a subset of the compact set
$$
\mathcal N^m_{h,M}:=\{ u e^h\bar u\inv\in Z_h^{>0} \ |\  
m\le a_i\le M  \quad\forall i\in I\,  \}\cong [m,M]^N.
$$
\vskip .3cm

Suppose indirectly that we have an index $i$ and a sequence 
$u_{(s)} e^h{\bar u_{(s)}}\inv$ in $\mathcal N_{h,M} $ 
for which the coordinate $a_{i}\to 0$
as $s\to\infty$. We may assume that $u_{(s)}\inv \cdot B_-$
converges, passing to a subsequence if necessary. Then it follows that 
$$ 
\lim_{s\to\infty} (u_{(s)}\inv \cdot B_-)\in B_-\dot w\cdot B_-
$$
for some $w< w_0$. On the other hand using $ u_{(s)}\inv \cdot B_-=e^h\bar u_{(s)}\inv \cdot B_+$
we see that
$$
\lim_{s\to\infty} (u_{(s)}\inv \cdot B_-)=e^h\lim_{s\to\infty} (\bar u_{(s)}\inv \cdot B_+)\in B_-\cdot B_+=B_-\dot w_0\cdot B_-.
$$
Namely, this last limit cannot leave the big cell $B_-\cdot B_+$  because the $b_i$ coordinates of the $\bar u_{(s)}$ are bounded from above by $M$. 

Therefore we have arrived at a contradiction and the
Claim is proved.
\end{proof}

\begin{cor}
For every $h\in\mathfrak h_\R$ the function $\mathcal F_h$ has a totally positive
critical point.
\end{cor} 

The totally positive critical point is provided by the minimum of $\mathcal F_{h}$ on $Z_h^{>0}$. 
 For type $A$ this result was proved already in \cite{Rie:TotPosGBCKS}, where 
moreover it was shown that the totally positive critical point is unique, and
this was used to describe the totally nonnegative part of the Peterson variety. 
The proof of uniqueness in the general case is joint work in preparation
with Thomas Lam \cite{LamRie:Satake}.

The same proof as above with the negative pinning also gives the following
\begin{cor}
For every $h\in\mathfrak h_\R$ the function $\mathcal F_h$ has a critical point
in $\Gamma_h^{(w_0)}$. 
\end{cor} 

This critical point is given by a maximum of $\mathcal F_h$ on $\Gamma_h^{(w_0)}$.
We might call it the totally negative critical point, and it is a feature of the 
full flag variety that a symmetry between totally positive and totally 
negative critical points exists, compare \cite{Rie:QCohGr,Rie:TotPosGBCKS}
and Lemma~5.3 in \cite{Rie:MSgen}. 

Assuming the result about uniqueness of the totally positive/negative critical points
in every fiber over $h\in\mathfrak h_\R$, it seems natural to think of 
the family $\Gamma^{(w_0)}$ of integration cycles  as associated to 
the family of totally negative critical points, and 
the family $\Gamma^{(1)}$ of integration cycles  as associated to 
the family of totally positive critical points (in both cases via the construction
used by Givental~\cite{Giv:MSFlag}). 
For $SL_2$ this is exactly
the case, by direct calculation.

\bibliographystyle{amsplain}

\def\cprime{$'$}
\providecommand{\bysame}{\leavevmode\hbox to3em{\hrulefill}\thinspace}
\providecommand{\MR}{\relax\ifhmode\unskip\space\fi MR }
\providecommand{\MRhref}[2]{%
  \href{http://www.ams.org/mathscinet-getitem?mr=#1}{#2}
}
\providecommand{\href}[2]{#2}

\end{document}